\documentclass[12pt, reqno]{amsart}
\usepackage{amsmath, amsthm, amscd, amsfonts, amssymb}
\usepackage[bookmarksnumbered]{hyperref}

\newtheorem{theorem}{Theorem}[section]

\newtheorem{corollary}[theorem]{Corollary}
\theoremstyle{definition}

\theoremstyle{remark}

\numberwithin{equation}{section}

\begin{document}

\title[n-Jordan homomorphisms]{n-Jordan homomorphisms }
\author[M. Eshaghi Gordji]{M. Eshaghi Gordji }
\address{Department of Mathematics, Semnan University, P. O. Box 35195-363, Semnan, Iran}
\email{madjid.eshaghi@gmail.com}



\begin{abstract}
Let $n\in \Bbb N,$  and let  $A,B$ be two rings. An additive map $h:
A\to B$ is called n-Jordan homomorphism if $h(a^n)=(h(a))^n$ for all
$a \in {A}$. Every Jordan homomorphism is an n-Jordan homomorphism,
for all $n\geq 2,$ but the converse is false, in general. In this
paper we investigate the n-Jordan homomorphisms  on Banach algebras.
Indeed some results related to continuity  are given as well.
\end{abstract}
\maketitle


\section{Introduction and preliminaries}

Let $A,B$ be two rings (algebras). An additive  map $h: A\to B$ is
called n-Jordan homomorphism (n-ring homomorphism) if
$h(a^n)=(h(a))^n$ for all $a \in {A},$ (
$h(\Pi^n_{i=1}a_i)=\Pi^n_{i=1}h(a_i),$ for all $a_1,a_2, \cdots,a_n
\in {A}$). If  $h: A\to B$ is a linear n-ring homomorphism, we say
that $h$ is n-homomorphism. The concept of n-homomorphisms was
studied for complex algebras by Hejazian, Mirzavaziri, and Moslehian
[3] (see also [1] and
 [7]). A 2-Jordan homomorphism is a Jordan homomorphism, in the
 usual  sense, between rings. Every Jordan homomorphism is an n-Jordan homomorphism,
for all $n\geq 2,$ (see for example Lemma 6.3.2 of [6]), but the
converse is false, in general. For instance, let $A$ be an algebra
over $\Bbb C$ and let $h:A \to A$ be a non-zero Jordan homomorphism
on $A$. Then $-h$ is a 3-Jordan homomorphism. It is easy to check
that $-h$ is not 2-Jordan homomorphism or 4-Jordan homomorphism. The
study of ring homomorphisms between Banach algebras $A$ and $B$ is
of interest even if $A=B=\Bbb C.$ For example the zero mapping, the
identity and the complex conjugate are ring homomorphisms on $\Bbb
C$, which are all continuous. On the other hand the existence of a
discontinuous ring homomorphism on $\Bbb C$ is well-known. More
explicitly, if $G$ is the set of all surjective ring homomorphisms
on $\Bbb C,$ then $Card(G)=2^{Card(\Bbb C)}.$ In fact, Charnow
[2;Theorem 3]   proved that there exist $2^{Card(\Bbb C)}$
automorphisms for every algebraically closed fild $K$. It is also
known that if $A$ is a uniform algebra on a compact metric space,
then there are exactly $2^{Card(\Bbb C)}$  complex-valued ring
homomorphisms on $A$ whose kernels are non-maximal prime ideals (see
[4; Corollary 2.4]). As an example, take
\[{\mathcal A} := \left[ \begin{array}{cccc}
{0} & {\Bbb R} & {\Bbb R} & {\Bbb R}\\
{0} & {0} & {\Bbb R} & {\Bbb R}\\
{0} & {0} & {0} & {\Bbb R}\\
{0} & {0} & {0} & {0}\\
 \end{array} \right] \]
then $\mathcal A$ is an  algebra equipped with the usual matrix-like
operations. It is easy to see that $$\mathcal A^3\neq 0=\mathcal
A^4.$$ So any additive map from $\mathcal A$ into itself is a
4-Jordan homomorphism, but its kernel does not need to be an ideal
of $\mathcal A.$ Let now $\mathcal B$ be the algebra of  all
$\mathcal A-$valued continuous functions from $[0,1]$ into $\mathcal
A$ with sup-norm. Then $\mathcal B$ is an infinite dimension Banach
algebra, also product of any four elements of $\mathcal B$ is 0.
Since $\mathcal B$ is infinite dimension, there are linear
discontinuous maps which are 4-Jordan homomorphisms from $\mathcal
B$ into itself (see [3]). In this paper we study the continuity of
linear n-Jordan homomorphisms on $C^*-$algebras.

\section{Main result}

By definition, it is obvious that n-ring homomorphisms are n-Jordan
homomorphisms. Conversely, under a certain condition, n-Jordan
homomorphisms are ring homomorphisms. For example, each Jordan
homomorphism $h$ from a commutative Banach algebra $A$ into $\Bbb C$
is a ring homomorphism: Fix $a,b\in A$ arbitrarily. Since
$h((a+b)^2)=(h(a+b)^2,$ a simple calculation shows that
$h(ab+ba)=2h(a)h(b).$ The commutativity of $A$ implies
$h(ab)=h(a)h(b),$ and hence $h$ is a ring homomorphism. In 1968,
Zelazko [8] proved the following Theorem (see also Theorem 1.1 of
[5]).

\begin{theorem}
Suppose $A$ is a Banach algebra,  which need not to be commutative,
and suppose $B$ is a semisimple commutative Banach algebra. Then
each Jordan homomorphism $h:\to  B$ is a ring homomorphism.
\end{theorem}
We prove the following result for 3-Jordan homomorphisms and
4-Jordan homomorphisms on commutative algebras.

\begin{theorem}
Let  $n\in \{3,4\}$ be fixed,  $ A, B$ be two commutative algebras,
and let $h: A\to B$ be a n-Jordan homomorphisms. Then $h$ is n-ring
homomorphism.
\end{theorem}
\begin{proof}
First let $n=3.$ Recall that $h$  is additive mapping such that
$h(a^3)=(h(a))^3$ for all $a\in A.$ Replacement of $a$ by $x+y$
results in
$$h(x^2y+xy^2)=h(x)^2h(y)+h(x)h(y)^2. \eqno (1)$$
Hence for every $x,y,z \in A$ we have
\begin{align*}
h(xyz)&=\frac{1}{2}h\{(x+z)^2y+(x+z)y^2-(x^2y+xy^2+z^2y+zy^2)\}\\
&=\frac{1}{2}\{h[(x+z)^2y+(x+z)y^2]-h[x^2y+xy^2]-h[z^2y+zy^2]\}\\
&=\frac{1}{2}\{[h(x+z)]^2h(y)+(x+z)[h(y)]^2-[h(x)]^2h(y)+h(x)[h(y)]^2\\
&-[h(z)]^2h(y)+h(z)[h(y)]^2\}\\
&=h(x)h(y)h(z).
\end{align*}
This means that $h$ is a 3-ring homomorphism. Now suppose $n=4.$
Then $h$ is additive and $h(a^4)=(h(a))^4$ for all $a\in A.$
Replacing  $a$ by $x+y$ in above equality to get
$$h(4x^3y+6x^2y^2+4xy^3)=4h(x)^3h(y)+6h(x)^2h(y)^2+4h(x)h(y)^3. \eqno (2)$$
Replacing $x$ by $x+z$ in (2),  we obtain
\begin{align*}h\{(4x^3y&+6x^2y^2+4xy^3)+(4z^3y+6z^2y^2+4zy^3)+12(x^2zy+xz^2y+xzy^2)\}\\
&=(4h(x)^3h(y)+6h(x)^2h(y)^2+4h(x)h(y)^3)+(4h(z)^3h(y)+6h(z)^2h(y)^2\\
&+4h(z)h(y)^3)+12(h(x)^2h(z)h(y)+h(x)h(z)^2h(y)+h(x)h(z)h(y)^2).
\hspace {1.8 cm} (3)
\end{align*}
Combining (2) by (3) to get
$$h\{(xyz)(x+y+z)\}=(h(x)h(y)h(z))(h(x)+h(y)+h(z)).\eqno (4)$$
Replacing $z$ by $-x$ in (4) to obtain
$$h(x^2y^2)=h(x)^2h(y)^2 \eqno (5)$$
replacing $y$ by $y+w$ in (5), we get
$$h(x^2yw)=h(x)^2h(y)h(w). \eqno(6)$$
Now replace $x$ by $x+t$ to obtain
$$h(xtyw)=h(x)h(t)h(y)h(w)$$
hence,  $h$ is 4-ring homomorphism.

\end{proof}
By Theorem 2.2 and Theorem 3.2 of [1], we conclude  the following
result.
\begin{corollary}
Let $h:A\to  B$ be a linear involution preserving 3-Jordan
homomorphism between commutative $C^*$-algebras. Then $h$ is norm
contractive ($\|h\|\leq 1$).
\end{corollary}
Also by above Theorem and Theorem 2.3 of [7], we have the following.
\begin{corollary}
Let $h:A\to  B$ be a linear involution preserving 4-Jordan
homomorphism between commutative $C^*$-algebras, then $h$ is
completely positive. Thus $h$ is bounded.
\end{corollary}
Now we prove our main Theorem.
\begin{theorem}
Suppose $A$ is a Banach algebra,  which need not to be commutative,
and suppose $B$ is a semisimple commutative Banach algebra. Then
each 3-Jordan homomorphism $h:A\to  B$ is a 3-ring homomorphism.
\end{theorem}
\begin{proof}
We prove the Theorem in two steps as follows. \\
STEP I. Suppose $B=\Bbb C.$  We have $h(a^3)=h(a)^3$ for all $a\in
A$. Replace $a$ by $x+y$ to obtain
$$h(xyx+yx^2+y^2x+x^2y+xy^2+yxy)=3(h(x)^2h(y)+h(x)h(y)^2) \eqno (7)$$
replace $y$ by $-y$ in (7) to get
$$h(-xyx-yx^2+y^2x-x^2y+xy^2+yxy)=3(-h(x)^2h(y)+h(x)h(y)^2). \eqno (8)$$
By (7), (8), we obtain the relation
$$h(xy^2+y^2x+yxy)=3(h(x)h(y)^2). \eqno (9)$$
Replacing $y$ by $y-z$ in (9), we get
\begin{align*}
h(xy^2&+xz^2-2xyz+yxy-yxz-zxy+zxz+z^2x+y^2x-2yzx)\\
&=3(h(x)^2h(y)+h(x)h(y)^2)-6h(x)h(y)h(z).\hspace{6 cm} (10)
\end{align*}
By (9) and (10), we obtain
$$h(yxz+zxy+2xyz+2yzx)=6h(x)h(y)h(z) \eqno (11)$$
replacing $z$ by $x$ in (11) to get
$$h(3yx^2+x^2y+2xyx=6h(x)^2h(y)\eqno (12)$$
combining (9) and (12) to obtain
$$h(xyx+2yx^2)=3h(x)^2h(y).\eqno (13)$$
By (8) and (13), we conclude that
$$h(yx^2-x^2y)=0. \eqno (14)$$
Replacing $x$ by $x+z$ in (14) to get
$$h(yx^2+yz^2+2yxz-x^2y-z^2y-2xzy)=0$$
by above equality and (14) it follows that
$$h(yxz-xzy)=0. \eqno (15)$$
Combining (11) and (15), we obtain
$$h(yxz+3xyz+2yzx)=6h(x)h(y)h(z) \eqno (16)$$
replace $z$ by $x$ in (16), we get
$$h(xyx+yx^2)=2h(x)^2h(y)  \eqno(17)$$
combining (13) and (17) to obtain
$$h(yx^2)=h(y)h(x)^2\eqno (18)$$
replace $x$ by $x+z$ in (18), we conclude that
$$h(yxz)=h(y)h(x)h(z)$$
hence $h$ is 3-ring homomorphism.\\
STEP II. $B$ is arbitrary semisimple and commutative. Let $M_B$ be
the maximal ideal space of $B$. We associate to each $f\in M_B$ a
function $h_f:A \to \Bbb C$ defined by
$$h_f(a):=f(h(a))$$
for all $a\in A.$ It is easy to see that $h_f$ is additive and
$h_f(a^3)=(h_f(a))^3$ for all $a\in A.$ So  STEP I applied to $h_f$,
implies that $h_f$ is a 3-ring homomorphism. By the definition of
$h_f$, we obtain that
$$f(h(abc))=f(h(a))f(h(b))f(h(c))=f(h(a)h(b)h(c)).$$
Hence $$h(abc)-h(a)h(b)h(c)\in Ker(f)$$ for all $a,b,c\in A$ and all
$f\in M_B.$  Since $B$ is assumed to be semisimple, we get
$h(abc)=h(a)h(b)h(c)$ for all $a,b,c \in A.$ We thus conclude that
$h$ is a 3-ring homomorphism, and the proof is complete.
\end{proof}

From now on we consider such  $n-$Jordan homomorphisms that are
linear.
\begin{corollary}
Suppose $A, B$ are $C^*-$algebras,  which $A$ need not to be
commutative, and suppose $B$ is  semisimple and commutative. Then
every  involution preserving  3-Jordan homomorphism $h:A\to  B$ is
norm contractive ($\|h\|\leq 1$).
\end{corollary}
\begin{proof}
It follows from Theorem 2.3 above and Theorem 2.1 of [1].
\end{proof}

\begin{theorem}
Let $h:A\to  B$ be a bounded  involution preserving k-Jordan
homomorphism between  $C^*$-algebras such that $h(a^*a)=h(a)^*h(a)$
for all $a\in A$. Then $h$ is norm contractive ($\|h\|\leq 1$).
\end{theorem}
\begin{proof}
By  using  Lemma 2.4 of [7], we have
\begin{align*}
\|h(a)\|^{4k+2}&=\|(h(a)^*h(a))^{2k+1}\|=\|(h(a)^*h(a))^{k}(h(a)^*h(a))(h(a)^*h(a))^{k}\|\\
&=\|[h(a)(h(a)^*h(a))^{k}]^*[h(a)(h(a)^*h(a))^{k}]\|\\
&=\|h(a)(h(a)^*h(a))^{k}\|^2 =\|h(a)(h(a^*a))^{k}\|^2\\
&=\|h(a)(h(a^*a)^{k})\|^2\leq \|h(a)\|^2\|h((a^*a)^k)\|^2\\
&\leq \|h\|^2\|a\|^2\|h\|^2\|(a^*a)^k\|^2\\
&\leq \|h\|^4\|a\|^{4k+2},
\end{align*}
for  all $a\in A.$ Which implies that $\|h\|\leq 1$ by taking
$4k+2-th$ roots.
\end{proof}

\end{document}